\documentclass[12pt]{amsart}
\usepackage{amssymb}
\usepackage{hyperref}
\usepackage{tabularx}
\usepackage{booktabs}
\usepackage{caption}

\newtheorem{thm}{Theorem}[section]
\newtheorem{lem}[thm]{Lemma}
\newtheorem{cor}[thm]{Corollary}
\newtheorem{prop}[thm]{Proposition}
\newtheorem{ex}[thm]{Example}

\newtheorem{con}[thm]{Conjecture}
\newtheorem*{prob*}{Open problem}

\theoremstyle{definition}

\newtheorem{defi}[thm]{Definition}

\theoremstyle{remark}

\newtheorem{rem}[thm]{Remark}
\newtheorem*{rem*}{Remark}

\DeclareMathOperator{\id}{id}

\newcommand{\kringel}{\mathbin{\raise1pt\hbox{$\scriptstyle\circ$}}}
\newcommand{\pkt}{\mathbin{\raise0pt\hbox{$\scriptstyle\bullet$}}}

\newcommand{\F}{\mathbb{F}}

\newcommand{\ad}{{\rm ad}}

\newcommand{\End}{{\rm End}}
\newcommand{\Der}{{\rm Der}}

\newcommand{\La}{\mathfrak{a}}

\newcommand{\Ld}{\mathfrak{d}}

\newcommand{\Lf}{\mathfrak{f}}
\newcommand{\Lg}{\mathfrak{g}}
\newcommand{\Lh}{\mathfrak{h}}

\newcommand{\Ln}{\mathfrak{n}}

\newcommand{\Lq}{\mathfrak{q}}

\newcommand{\al}{\alpha}
\newcommand{\be}{\beta}
\newcommand{\ga}{\gamma}
\newcommand{\de}{\delta}
\newcommand{\ep}{\varepsilon}
\newcommand{\ka}{\kappa}

\newcommand{\ov}{\overline}

\newcommand{\ra}{\rightarrow}

\renewcommand{\phi}{\varphi}

\begin{document}


\title[CPA-structures Structures]{Commutative post-Lie algebra structures and linear equations for 
nilpotent Lie algebras}

\author[D. Burde]{Dietrich Burde}
\author[W. Moens]{Wolfgang Alexander Moens}
\author[K. Dekimpe]{Karel Dekimpe}
\address{Fakult\"at f\"ur Mathematik\\
Universit\"at Wien\\
  Oskar-Morgenstern-Platz 1\\
  1090 Wien \\
  Austria}
\email{dietrich.burde@univie.ac.at}
\address{Fakult\"at f\"ur Mathematik\\
Universit\"at Wien\\
  Oskar-Morgenstern-Platz 1\\
  1090 Wien \\
  Austria}
\email{wolfgang.moens@univie.ac.at}
\address{Katholieke Universiteit Leuven Kulak\\
E. Sabbelaan 53 bus 7657\\
8500 Kortrijk\\
Belgium}
\email{karel.dekimpe@kuleuven-kulak.be}

\date{\today}


\begin{abstract}
We show that for a given nilpotent Lie algebra $\Lg$ with $Z(\Lg)\subseteq [\Lg,\Lg]$ all
commutative post-Lie algebra structures, or CPA-structures, on $\Lg$ are complete. This means that
all left and all right multiplication operators in the algebra are nilpotent. Then we study 
CPA-structures on free-nilpotent Lie algebras $F_{g,c}$ and discover a strong relationship
to solving systems of linear equations of type $[x,u]+[y,v]=0$ for generator pairs $x,y\in F_{g,c}$.
We use results of Remeslennikov and St\"ohr concerning these equations to prove that, for certain $g$ and 
$c$, the free-nilpotent Lie algebra $F_{g,c}$ has only central CPA-structures.
 \end{abstract}

\maketitle

\section{Introduction}

Post-Lie algebras and post-Lie algebra structures arise in many areas of mathematics and physics.
They have been studied by Vallette \cite{VAL} in connection with the homology of partition posets 
and the study of Koszul operads, by Loday \cite{LOD} within the context of algebraic operad triples,
and by several authors in connection with modified Yang-Baxter equations, double Lie algebras, $R$-matrices,
isospectral flows, Lie-Butcher series and other topics \cite{ELM}.
In our study of geometric structures on Lie groups, post-Lie algebras arise as a natural common generalization 
of pre-Lie algebras \cite{HEL,KIM,SEG,BU5,BU19,BU24} and LR-algebras \cite{BU34, BU38}, in the context of nil-affine 
actions of Lie groups. We have obtained several results on the existence of post-Lie algebra structures 
in general \cite{BU33,BU41,BU44,BU51}. A particular interesting class of post-Lie algebra structures
are {\em commutative} structures, so-called {\em CPA-structures}. We have studied CPA-structures mainly for
semisimple, perfect and complete Lie algebras in \cite{BU51, BU52}. In particular we have shown that 
CPA-structures on perfect Lie algebras are trivial. This is far from true for nilpotent Lie algebras.
Here the first remarkable observation was, that all CPA-structures on non-abelian $2$-generated nilpotent 
Lie algebras are {\it complete}, i.e., that all left and right multiplication operators $L(x)$ and $R(x)$
in the algebra are nilpotent \cite{BU51}. We have conjectured that this is true for all
nilpotent Lie algebras $\Lg$ without abelian factor. Here we will give a proof now, see Theorem $\ref{3.2}$.
This is a main result of this paper, which has further consequences concerning CPA-structures on nilpotent
Lie algebras. It follows, among other things, that $\Lg\cdot Z(\Lg)=0$ for CPA-structures on nilpotent Lie algebras 
with $1$-dimensional center $Z(\Lg)$, contained in $[\Lg,\Lg]$. \\
The second topic of this paper concerns the classification of CPA-structures on the free-nilpotent 
Lie algebra $F_{g,c}$, with $g\ge 2$ generators, and of nilpotency class $c\ge 2$. Here we find a surprising connection
to the work of Remeslennikov and St\"ohr \cite{RES} on the equation $[x,u]+[y,v]=0$ over free Lie algebras,
where $x$ and $y$ are free generators. We can apply their results to prove that the free-nilpotent Lie algebra
$F_{3,c}$ admits only central CPA-structures for all $c\ge 3$. The argument also works for $F_{g,c}$ with given
$g\ge 4$ for all $c\ge 3$, provided we can show the result for $F_{g,3}$. Here we conjecture that all CPA-structures
on $F_{g,c}$ for $g\ge 3$ and $g\ge 2$ are central, see Conjecture $\ref{4.9}$. For $g=2$ this leads to the study of
the system of linear equations
\begin{align*}
[x,u]+[y,v] & = 0,\\
[x,v]+[y,w] & = 0,
\end{align*}
where $(x,y)$ is a generating pair of $F_{2,c}$. We want to find all solutions $u,v,w\in [\Lg,\Lg]$.
The conjecture is that they are all contained in the center of $F_{2,c}$.
In the last section, we study such equations for general $2$-generated nilpotent Lie algebras.
We say that such a Lie algebra {\em has property $F$}, if for every generating pair $(x,y)$ of $F_{2,c}$,
all solutions $u,v,w\in [\Lg,\Lg]$ of the above equations are central. We consider the invariant 
$z(\Lg)=\frac{\dim Z(\Lg)}{\dim \Lg}$, which gives a necessary condition for $2$-generated nilpotent Lie 
algebras $\Lg$ to have property $F$. In fact, if $Z(\Lg)<\frac{1}{3}$, then $\Lg$ cannot have property $F$, 
see Lemma $\ref{5.3}$. We conjecture that the only such Lie algebras having property $F$ are the free-nilpotent
Lie algebras $F_{2,c}$, for $c\ge 3$.

\section{Preliminaries}

Let $K$ denote a field of characteristic zero, if not said otherwise, and $\Lg$ be a finite-dimensional
Lie algebra over $K$. Denote by $\Lg^1=\Lg$, $\Lg^{i+1}=[\Lg,\Lg^i]$ for $i\ge 1$ the lower central series 
ideals of $\Lg$, and by $\Lg^{(1)}=\Lg$, $\Lg^{(i+1)}=[\Lg^{(i)},\Lg^{(i)}]$ for $i\ge 1$ the derived ideals.

\begin{defi}
A Lie algebra $\Lg$ is called a {\it stem Lie algebra}, if the center is contained in the commutator,
i.e., if $Z(\Lg)\subseteq [\Lg,\Lg]$. 
\end{defi}

An indecomposable Lie algebra is a stem Lie algebra, but the converse need not hold. \\[0.2cm]
Let us first recall the definition of post-Lie algebra structures on pairs of 
Lie algebras $(\Lg,\Ln)$ over $K$ \cite{BU41}:

\begin{defi}\label{pls}
Let $\Lg=(V, [\, ,])$ and $\Ln=(V, \{\, ,\})$ be two Lie brackets on a vector space $V$ over 
$K$. A {\it post-Lie algebra structure}, or {\em PA-structure} on the pair $(\Lg,\Ln)$ is a 
$K$-bilinear product $x\cdot y$ satisfying the identities:
\begin{align}
x\cdot y -y\cdot x & = [x,y]-\{x,y\} \label{post1}\\
[x,y]\cdot z & = x\cdot (y\cdot z) -y\cdot (x\cdot z) \label{post2}\\
x\cdot \{y,z\} & = \{x\cdot y,z\}+\{y,x\cdot z\} \label{post3}
\end{align}
for all $x,y,z \in V$.
\end{defi}

Define by  $L(x)(y)=x\cdot y$ and $R(x)(y)=y\cdot x$ the left respectively right multiplication 
operators of the algebra $A=(V,\cdot)$. By \eqref{post3}, all $L(x)$ are derivations of the Lie 
algebra $(V,\{,\})$. Moreover, by \eqref{post2}, the left multiplication
\[
L\colon \Lg\ra \Der(\Ln)\subseteq \End (V),\; x\mapsto L(x)
\]
is a linear representation of $\Lg$. A particular case of a post-Lie algebra structure arises 
if the algebra $A=(V,\cdot)$ is {\it commutative}, i.e., if $x\cdot y=y\cdot x$ is satisfied for all 
$x,y\in V$, so that we have $L(x)=R(x)$ for all $x\in V$. Then the two Lie brackets $[x,y]=\{x,y\}$ 
coincide, and we obtain a commutative algebra structure on $V$ associated with only one Lie 
algebra \cite{BU51}:

\begin{defi}\label{cpa}
A {\it commutative post-Lie algebra structure}, or {\em CPA-structure} on a Lie algebra $\Lg$ 
is a $K$-bilinear product $x\cdot y$ satisfying the identities:
\begin{align}
x\cdot y & =y\cdot x \label{com4}\\
[x,y]\cdot z & = x\cdot (y\cdot z) -y\cdot (x\cdot z)\label{com5} \\
x\cdot [y,z] & = [x\cdot y,z]+[y,x\cdot z] \label{com6}
\end{align}
for all $x,y,z \in V$. 
\end{defi}

Let $\Ld$ be a nilpotent Lie subalgebra of the derivation algebra $\Der(\Lg)$. 
For a linear functional $\al\in \Ld^{\ast}$ we define the vector subspace
\[
\Lg_{\al}=\{v \in \Lg\mid (d-\al(d)\id_{\Lg})^{\dim(\Lg)}(v)=0 \text{ for all }d\in \Ld\}.
\]
If $\Lg_{\al}$ is nonzero, then $\al\in \Ld^{\ast}$ is called a {\em weight} of $\Lg$ with respect to $\Ld$, and 
$\Lg_{\al}$ is called a {\em weight space}. The following result is well-known: 

\begin{thm}[Weight space decomposition]\label{2.2}
Let $\Lg$ be a Lie algebra over an alegbraically closed field of characteristic zero and $\Ld$ be a nilpotent 
Lie subalgebra of $\Der(\Lg)$. Then $\Lg$ decomposes as a direct vector space sum
\[
\Lg=\bigoplus_{\al\in \Ld^*}\Lg_{\al},
\]
and for all $\al,\be\in \Ld^*$ we have
\[
[\Lg_{\al},\Lg_{\be}]\subseteq \Lg_{\al+\be}.
\]
\end{thm} 

Note that the weight space decomposition makes sense for post-Lie algebra structures on
pairs $(\Lg,\Ln)$ as follows. Suppose that $\Lg$ is nilpotent. Then the left multiplication map
$L\colon \Lg\ra \Der(\Ln)$ is a Lie algebra homomorphism, so that we may specialize Theorem $\ref{2.2}$ 
to
\[
\Ld:=L(\Lg)\subseteq \Der(\Ln),
\]
which is a nilpotent subalgebra of derivations.

\section{Complete CPA-structures}

In accordance with complete LR-structures \cite{BU38} and complete pre-Lie algebra structures \cite{SEG} we
define complete CPA-structures as follows.

\begin{defi}
A CPA-structure on a Lie algebra $\Lg$ is called {\em complete}, if all left multiplication operators
$L(x)$, and hence all right multiplication operators $R(x)$, are nilpotent.
\end{defi}

We have shown in \cite{BU51} that all CPA-structures on a non-abelian nilpotent Lie algebra 
generated by two elements are complete. The question then was, whether this is true for all
nilpotent Lie algebras without an abelian factor. The aim of this section is to give a positive answer
to this question. We will prove it in Theorem $\ref{3.2}$. However, first
we need two lemmas.

\begin{lem}\label{3.3}
Let $L$ be a nilpotent Lie algebra, and $(A,\cdot)$ a CPA-structure on $\Lg$. Then
for all weights $\al,\be\in \Lg^{\ast}$ we have the inclusion
\[
\Lg_{\al}\cdot\Lg_{\be}\subseteq \Lg_{\al}\cap \Lg_{\be}.
\]
Furthermore, if $\al\neq \be$ we have
\[
\al (\Lg_{\be})=\be (\Lg_{\al})=0.
\]
\end{lem}

\begin{proof}
By construction the weight spaces $\Lg_{\al}$ and $\Lg_{\be}$ are invariant under left multiplication, i.e., 
we have $L(x)\Lg_{\al}\subseteq \Lg_{\al}$ and $L(y)\Lg_{\be}\subseteq \Lg_{\be}$ for all $x,y\in A$.
By specializing $x\in \Lg_{\be}$ and $y\in \Lg_{\al}$ we obtain
\begin{align*}
\Lg_{\al}\cdot \Lg_{\be} & \subseteq \Lg_{\be},\\
\Lg_{\be}\cdot \Lg_{\al} & \subseteq \Lg_{\al}.
\end{align*}
Since the CPA-structure is commutative we have $\Lg_{\al}\cdot \Lg_{\be}=\Lg_{\be}\cdot \Lg_{\al}$,
which gives $\Lg_{\al}\cdot\Lg_{\be}\subseteq \Lg_{\al}\cap \Lg_{\be}$. \\
For $\al\neq \be$ we have
\[
L(\Lg_{\al})\Lg_{\be}=\Lg_{\al}\cdot \Lg_{\be}\subseteq \Lg_{\al}\cap \Lg_{\be}=0.
\]
Hence all operators in $L(\Lg_{\al})$ vanish on $\Lg_{\be}$, so that in particular the diagonal
components $\be(\Lg_{\al})$ in $L(\Lg_{\al})_{\mid \Lg_{\be}}$ vanish identically. This gives
$\be(\Lg_{\al})=0$. By symmetry we obtain $\al(\Lg_{\be})=0$.
\end{proof}

\begin{lem}\label{3.4}
Let $\Lg$ be a nilpotent Lie algebra with a CPA-structure $(A,\cdot)$. Then the weight space
decomposition induced by $L(\Lg)\subseteq \Der(\Lg)$ satisfies
\[
[\Lg,\Lg]\subseteq \Lg_0.
\]
\end{lem}

\begin{proof}
By Lemma $\ref{3.3}$ we have $\Lg_0\cdot \Lg_0\subseteq \Lg_0$. We first show that
\[
 \Lg\cdot [\Lg,\Lg]\subseteq \Lg_0\cdot (\Lg_0\cdot \Lg_0)\subseteq \Lg_0.
\]
By linearity, it is sufficient to show this for weight spaces, i.e., to show that
\[
\Lg_{\al}\cdot [\Lg_{\be},\Lg_{\ga}]\subseteq \Lg_0\cdot (\Lg_0\cdot \Lg_0)\subseteq \Lg_0.
\]
For $\Lg_{\al}\cdot [\Lg_{\be},\Lg_{\ga}]= 0$ there is nothing to prove, so that we may assume
that 
\[
0\neq \Lg_{\al}\cdot [\Lg_{\be},\Lg_{\ga}]\subseteq \Lg_{\al}\cdot \Lg_{\be+\ga}\subseteq \Lg_{\al}\cap \Lg_{\be+\ga}
\]
by the grading property of the weight space decomposition and by Lemma $\ref{3.3}$. Hence we have 
$\al=\be+\ga$. By \eqref{com4} and \eqref{com5} we obtain
\begin{align*} 
\Lg_{\al}\cdot [\Lg_{\be},\Lg_{\ga}] & = \Lg_{\be+\ga}\cdot [\Lg_{\be},\Lg_{\ga}] \\
 & = [\Lg_{\be},\Lg_{\ga}]\cdot \Lg_{\be+\ga} \\
 & \subseteq \Lg_{\be}\cdot (\Lg_{\ga}\cdot \Lg_{\be+\ga})-\Lg_{\ga}\cdot (\Lg_{\be}\cdot \Lg_{\be+\ga}).
\end{align*}
Since the left-hand side is nonzero, one of the right-hand side terms must be nonzero as well.
Whichever it is, Lemma $\ref{3.3}$ implies that $\be=\ga=\be+\ga$, so that
\[
\al=\be=\ga=0.
\]
Hence we have shown that $\Lg\cdot [\Lg,\Lg]\subseteq \Lg_0$, or $L(\Lg)([\Lg,\Lg])\subseteq \Lg_0$.
By induction we obtain $L(\Lg)^{n+1}([\Lg,\Lg])\subseteq L(\Lg)^n(\Lg_0)$ for all $n\ge 1$. By definition of
$\Lg_0$ there is some $n$, for example $\dim(\Lg)$, such that $L(\Lg)^n(\Lg_0)=0$. Hence $[\Lg,\Lg]\subseteq \Lg_0$.
\end{proof}

The above lemma has already interesting consequences.

\begin{cor}
Let $\Lg$ be a nilpotent Lie algebra with CPA-structure $(A,\cdot)$ and left multiplication
map $L\colon \Lg\ra \Der(\Lg)$. Let $\La\subseteq \Lg^t$ be a characteristic ideal of $\Lg$ contained
in the lower central series ideal $\Lg^t$, and $r=\lceil \frac{\dim(\La)+t-1}{t} \rceil$. 
Then we have
\[
L(\La)^r(\Lg)=0.
\]
In particular, if $\Lg$ is a nilpotent stem Lie algebra, i.e., satisfying $Z(\Lg)\subseteq \Lg^2=[\Lg,\Lg]$ 
we have $r=\lceil \frac{\dim Z(\Lg)+1}{2}\rceil $ and $L(Z(\Lg))^r(\Lg)=0$. This gives 
$L(Z(\Lg))(\Lg)=\Lg\cdot Z(\Lg)=0$ for $\dim Z(\Lg)=1$.
\end{cor} 

\begin{proof}
By Lemma $\ref{3.4}$ we have $L(\Lg)^m(\La)=0$ for $m=\dim (\La)$. This yields
\begin{align*}
L(\La)^{\lceil \frac{m+t-1}{t}\rceil}(\Lg) & = L(\La)^{\lceil \frac{m-1}{t}\rceil}(\La\cdot \Lg)\\
 & = L(\La)^{\lceil \frac{m-1}{t}\rceil}(\Lg\cdot \La) \\
 & \subseteq L(\Lg)^m(\La)=0.
\end{align*}
\end{proof}

\begin{ex}\label{3.6}
Let $\Lg$ be the $(2m+1)$-dimensional Heisenberg Lie algebra. Then $Z(\Lg)=[\Lg,\Lg]$ is $1$-dimensional,
and for any CPA-structure on $\Lg$ we have $\Lg\cdot [\Lg,\Lg]=\Lg\cdot Z(\Lg)=0$.
\end{ex}

Now we can prove the following main theorem.

\begin{thm}\label{3.2}
Let $\Lg$ be a nilpotent stem Lie algebra. Then every CPA-structure on $\Lg$ is complete.
\end{thm}

\begin{proof}
We can naturally embed $\Lg$ into the nilpotent Lie algebra $\Lg_{\ov{K}}:=\Lg\otimes_K\ov{K}$, which is again a 
stem Lie algebra. Furthermore the post-Lie algebra structure on $\Lg$ can be naturally extended to
$\Lg_{\ov{K}}$. In particular, if the structure on  $\Lg_{\ov{K}}$ is complete then also the structure on
$\Lg$ is complete. So we may assume that $K$ is algebraically closed.
So we can use the weight space decomposition. Now we only need to show that $\Lg=\Lg_0$. For this we fix
a weight $\al\in \Lg^{\ast}$, and then show that it is the zero-weight. By Lemma $\ref{3.3}$, $\al$ 
vanishes on all weight spaces different from $\Lg_{\al}$:
\[
\al\bigl(\bigoplus_{\ga\neq \al}L_{\ga} \bigr)=0.
\] 
By Lemma $\ref{3.4}$ we have $[\Lg,\Lg]\subseteq \Lg_0$. Suppose that $\al\neq 0$.
Then $\Lg_{\al}$ has trivial intersection with $\Lg_0$, and therefore trivial intersection
with $Z(\Lg)\subseteq [\Lg,\Lg]\subseteq \Lg_0$. Hence there exists some weight $\beta$ such that
\[
0\neq [\Lg_{\al},\Lg_{\be}]\subseteq \Lg_{\al+\be}\cap[\Lg,\Lg]\subseteq \Lg_{\al+\be}\cap \Lg_0.
\]
This forces $\al+\be=0$. Since $K$ has characteristic zero, $-\al\neq \al$. Hence Lemma $\ref{3.3}$ implies 
that 
\[
-\al(\Lg_{\al})=\be(\Lg_{\al})=0. 
\]
This yields
\[
\al(\Lg)=\al\bigl(\bigoplus_{\ga}L_{\ga} \bigr)=\al\bigl(\bigoplus_{\ga\neq \al}L_{\ga} \bigr)\oplus \al(\Lg_{\al})=0.
\]
This is a contradiction, and we are done. 
\end{proof}

\begin{rem}
The proof can be extended to stem nilpotent Lie algebras over fields of prime characteristic $p>2$.
The only argument depending on the characteristic was that we needed $-\al\neq \al$ to hold. 
In characteristic $2$ however, the result is no longer true. Consider
the $3$-dimensional Heisenberg Lie algebra $\Lh$ with basis $(x,y,z)$ and Lie bracket $[x,y]=z$.
Then $Z(\Lh)\subseteq [\Lh,\Lh]$, but there exist post-Lie algebra structures, which are not complete:
\[
x\cdot x=x,\; x\cdot y=y\cdot x=y,
\]
and all other products of basis elements equal to zero. Then $L(x)$ is not nilpotent.
\end{rem}

\begin{rem}
For a nilpotent Lie algebra $\Lg$ we even have equivalence: $\Lg$ is a stem Lie algebra if and only if all 
CPA-structures on $\Lg$ are complete. Indeed, suppose that $Z\Lg)$ is {\em not} contained 
in $[\Lg,\Lg]$. Then there exists $v\neq 0$ in $\Lg$ and an ideal $\La$ in $\Lg$ such that
\[
\Lg=K\cdot v \oplus \La,
\]
with $[v,\La]=0$. Then 
\[
v\cdot v=v,\; v\cdot \La=\La\cdot \La=0
\]
defines a CPA-structure on $\Lg$ which is not complete.
\end{rem}

CPA-structures on stem nilpotent Lie algebras sometimes satisfy other conditions in addition
to completeness.

\begin{defi}
Let $\Lg$ be a nilpotent Lie algebra. A CPA-structure on $\Lg$ is called {\em central}, if
it satisfies
\[
\Lg\cdot \Lg\subseteq Z(\Lg).
\]
\end{defi}

\begin{lem}
Let $\Lg$ be a nilpotent stem Lie algebra. A central CPA-structure on $\Lg$ satisfies
\[
\Lg\cdot Z(\Lg)=\Lg\cdot [\Lg,\Lg]=0.
\]
\end{lem}

\begin{proof}
Using $Z(\Lg)\subseteq [\Lg,\Lg]$, $\Lg\cdot \Lg\subseteq Z(\Lg)$ and \eqref{com5} we have
\begin{align*}
Z(\Lg)\cdot \Lg & \subseteq [\Lg,\Lg]\cdot \Lg \\
 & \subseteq [Z(\Lg),\Lg]+[\Lg,Z(\Lg)] \\
 & = 0.
\end{align*}
By commutativity also $\Lg\cdot Z(\Lg)=\Lg\cdot [\Lg,\Lg]=0$.
\end{proof}

Denote by $F_{2,3}$ the free-nilpotent Lie algebra with $2$ generators $e_1,e_2$ and nilpotency
class $3$. It has dimension $5$, with a basis $(e_1,\ldots ,e_5)=(e_1,e_2,[e_1,e_2], [e_1,[e_1,e_2]], 
[e_2,[e_1,e_2]])$ and brackets
\[
[e_1,e_2]=e_3,\; [e_1,e_3]=e_4,\;[e_2,e_3]=e_5.
\]

\begin{ex}\label{3.11}
All CPA-structures on $F_{2,3}$ are central, given by
\begin{align*}
e_1\cdot e_1 & = \al e_4+\be e_5,\\
e_1\cdot e_2 & = \ga e_4+\de e_5,\\
e_2\cdot e_2 & = \ep e_4+\ka e_5,
\end{align*}
for $\al,\be,\ga,\de,\ep,\ka\in K$. Of course, $e_2\cdot e_1=e_1\cdot e_2$.
\end{ex}

This follows by a direct computation. Note that not all CPA-structures on $\Lh=F_{2,2}$, 
the Heisenberg Lie algebra, are central.
Indeed, for $[e_1,e_2]=e_3$ we have a CPA-product given by
\[
e_1\cdot e_1=e_2,\; e_1\cdot e_2=\al e_3,
\]
which is not central. On the other hand we know that $\Lh\cdot Z(\Lh)=\Lh\cdot [\Lh,\Lh]=0$
by Example $\ref{3.6}$. Of course there are also examples where even these properties do not
hold. Consider the free-nilpotent Lie algebra $F_{3,2}$ with $3$ generators $e_1,e_2,e_3$ of
nilpotency class $2$. We have a basis $(e_1,\ldots ,e_6)$ with Lie brackets
\[
[e_1,e_2]=e_4,\; [e_1,e_3]=e_5,\;[e_2,e_3]=e_6.
\]
We have $Z(\Lg)=[\Lg,\Lg]$.

\begin{ex}\label{3.12}
There exist CPA-structures on $F_{3,2}$ with $\Lg\cdot Z(\Lg)\neq 0$.
\end{ex}

Indeed, consider the following CPA-product
\begin{align*}
e_1\cdot e_1 & = e_2,\\
e_1\cdot e_2 & = -e_5,\\
e_1\cdot e_5 & = e_6,\\
e_2\cdot e_3 & = -2e_6.
\end{align*}

\section{CPA-structures on free-nilpotent Lie algebras}

Denote by $F_{g,c}$ the free-nilpotent Lie algebra with $g$ generators and nilpotency class $c$.
We will assume that $g,c\ge 2$, because CPA-structures on abelian Lie algebras just
correspond to commutative and associative algebra structures on the underlying vector space,
see Example $2.3$ in \cite{BU52}. If $F_g$ denotes the free Lie algebra on $g$ generators
$\{x_1,\ldots,x_g \}$, $F_{g,c}$ can be defined by $F_g/F_g^{c+1}$. 
We have $Z(F_{g,c})\subseteq [F_{g,c},F_{g,c}]$, and
\[
\dim (F_{g,c})=\sum_{m=1}^c \frac{1}{m}\sum_{d\mid m}\mu(d)g^{\frac{m}{d}}.
\]
The last summand, for $m=c$ gives the dimension of $Z(F_{g,c})$. We first note
that $F_{g,c}$ always admits all possible central solutions.

\begin{lem}
Let $X=\{x_1,\ldots ,x_g\}$ be generators of  $F_{g,c}$ and $(z_1,\ldots,z_r)$ be a basis
of $Z(F_{g,c})$. Then
\[
x_i\cdot x_j=x_j\cdot x_i=\sum_{k=1}^r \al_{ij}^kz_k
\]
for $1\le i,j\le g$ and arbitary scalars $ \al_{ij}^k$ defines a central CPA-structure on $F_{g,c}$.
\end{lem}

\begin{proof}
Let $\Lg=F_{g,c}$. Because only products between generators from $X$ may have nonzero product, and 
both $Z(\Lg)$ and $[\Lg,\Lg]$ are not contained in $X$, we have
\[
\Lg\cdot (\Lg\cdot \Lg)\subseteq \Lg\cdot Z(\Lg)\subseteq \Lg\cdot [\Lg,\Lg]=0.
\]
Hence all axioms \eqref{com4}, \eqref{com5} and \eqref{com6} are satisfied.
\end{proof}

Our aim is to show that $F_{g,c}$ admits {\em only} these canonical central CPA-structures, for 
$c$ big enough. For $c=2$ we already know that this is not true, e.g., for $F_{2,2}$ and $F_{3,2}$, 
see Example $\ref{3.12}$. Hence it seems reasonable to assume that $c\ge 3$. Example $\ref{3.11}$ shows 
that $F_{2,3}$ only admits central CPA-structures. We will show that the same is true for $F_{3,3}$.
Let $X=\{x_1,x_2,x_3\}$. Then by constructing a Hall basis we find a basis $(x_1,\ldots ,x_{14})$ 
for $F_{3,3}$ with Lie brackets

\begin{align*}
[x_1,x_2] & = x_4,\, [x_1,x_3] = x_5,\, [x_1,x_4] = x_7, \, [x_1,x_5] = x_8,\\
[x_1,x_6] & = x_9,\, [x_2,x_3] = x_6,\, [x_2,x_4] = x_{10}, \, [x_2,x_5] = x_{11},  \\
[x_2,x_6] & = x_{12},\, [x_3,x_4] = x_{11}-x_9,\, [x_3,x_5] = x_{13}, \, [x_3,x_6] = x_{14}.
\end{align*}

\begin{prop}\label{4.2}
All CPA-structures on $F_{3,3}$ are central.
\end{prop}

\begin{proof}
For $k=1,\ldots ,14$ let $L(x_k)=(x_{ij}^k)$ be the left multiplication operators with variables $x_{ij}^k$. 
They are determined by $L(x_1)$ and $L(x_2)$. We know that
all $L(x)$ are nilpotent by Theorem $\ref{3.2}$. Furthermore all $L(x)$ are derivations of
of $F_{3,3}$, and these are known. Now we can solve the equations in $x_{ij}^k$, given by the axioms
by a direct computation. In fact, axioms \eqref{com4} and \eqref{com6} give linear equations, 
which are easy to solve. Then it remains to solve the quadratic equations given by
\eqref{com5}. Since many coefficients are already zero, we only have relatively few equations, so that
a computation with Gr\"obner bases yields the result.
\end{proof}

The idea now is to use induction on the nilpotency class $c$, and to show the following result:

\begin{thm}\label{4.3}
All CPA-structures on $F_{3,c}$ with $c\ge 3$ are central. 
\end{thm}

However, for the induction step to work, we need to solve equations
\[
[x,u]+[y,v]=0
\]
for generators $x,y$ and elements $u,v$ in $F_{g,c}$. As it turns out, there
are results in the literature concerning such equations in free Lie algebras.
The following theorem is due to Remeslennikov and St\"ohr \cite{RES}:

\begin{thm}\label{4.4}
Let $F$ be the free Lie algebra on the set $X=\{x_1,x_2,\ldots \}$. If $u,v\in F$ satisfy the linear equation
\[
[x_1,u]+[x_2,v]=0,
\]
then $u$ and $v$ are inner solutions, i.e., they are contained in the free subalgebra of $F$ generated
by $x_1$ and $x_2$.
\end{thm}

This result has the following corollary for $g\ge 3$:

\begin{cor}\label{4.5}
Let $F_g$ be the free Lie algebra with $g\ge 3$ generators $x_1,\ldots ,x_g$. Consider the linear system
of equations
\[
[u_{ij},x_k]+[x_j,u_{ik}]=0
\]
for all $1\le i,j,k\le g$ in the variables $u_{ij}$ with $u_{ij}=u_{ji}$ for all $i,j$.
Then all solutions $u_{ij}$ in the commutator $[F_g,F_g]$ are contained in the center $Z(F_g)=0$.
\end{cor}

\begin{proof}
Fix $i,j$ in $\{1,2,\ldots,g\}$ with $u_{ij}\in [F_g,F_g]$. We will show that $u_{ij}=0$.
Since $g\ge 3$ we may select a $k\in \{1,2,\ldots,g\}$, which is distinct from $i$ and $j$.
Then we have
\begin{align*}
[u_{ik},x_j] + [x_k,u_{ij}] & =0\\
[u_{ji},x_k] + [x_i,u_{jk}] & =0.
\end{align*}
By Theorem $\ref{4.4}$ it follows that $u_{ij}=u_{ji}$ is contained in the Lie algebra
\[
\langle x_i,x_k\rangle \cap \langle x_j,x_k \rangle = K\cdot x_k.
\]
Since we assume that $u_{ij}\in [F_g,F_g]$ we have
\[
u_{ij}\in K\cdot x_k \cap [F_g,F_g]=0.
\]
\end{proof}

This gives an analogous result for free-nilpotent Lie algebras.

\begin{cor}\label{4.6}
Let $F_{g,c}$ be the free Lie algebra with $g\ge 3$ generators $x_1,\ldots ,x_g$ of nilpoteny
class $c\ge 2$. Consider the linear system
of equations
\[
[u_{ij},x_k]+[x_j,u_{ik}]=0
\]
for all $1\le i,j,k\le g$ in the variables $u_{ij}$ with $u_{ij}=u_{ji}$ for all $i,j$.
Then all solutions $u_{ij}$ in the commutator $[F_{g,c},F_{g,c}]$ are contained in the center $Z(F_{g,c})$.
\end{cor}

\begin{proof}
Let $(u_{ij})$ be a solution in $[F_{g,c},F_{g,c}]$, and decompose each $u_{ij}$ into
its homogeneous components:
\[
u_{ij}=\sum_{2\le k\le c} u_{ij}^k.
\]
Then for each $2\le k\le c$ we obtain a new solution $(u_{ij}^k)$ of degree $k$ in $F_{g,c}$. 
For $k\le c-1$ the obvious lift of $(u_{ij}^k)$ is also a solution in the free Lie algebra $F_g$, because
the vanishing condition is of degree $k+1\le c$. By Corollary $\ref{4.5}$ we obtain
\[
u_{ij}^2=\cdots =u_{ij}^{c-1}=0.
\]
This means that all of the $u_{ij}$ are contained in the center of $F_{g,c}$. 
\end{proof}

\begin{ex}\label{4.7}
For $g=2$, generators $x_1,x_2$ and variables $u_{11},u_{12},u_{22}$ the 
linear system of equations, without repetitions, is given by
\begin{align*}
[u_{11},x_2]+[x_1,u_{12}] & = 0,\\
[u_{22},x_1]+[x_2,u_{12}] & = 0. \\
\end{align*}
For $g=3$, generators $x_1,x_2,x_3$ and variables $u_{11},u_{12},u_{13},u_{22},u_{23},u_{33}$ the 
linear system of equations, without repetitions, is given by
\begin{align*}
[u_{11},x_2]+[x_1,u_{12}] & = [u_{11},x_3]+[x_1,u_{13}] = 0,\\
[u_{12},x_3]+[x_2,u_{13}] & = [u_{22},x_1]+[x_2,u_{12}] = 0,\\
[u_{22},x_3]+[x_2,u_{23}] & = [u_{23},x_1]+[x_3,u_{12}] = 0,\\
[u_{33},x_1]+[x_3,u_{13}] & = [u_{33},x_2]+[x_3,u_{12}] = 0.\\
\end{align*}
\end{ex}

{\em Proof of Theorem $\ref{4.3}$}: We want to show by induction on $c\ge 3$ that
all CPA-structures on $\Lf:=F_{3,c}$ satisfy $\Lf\cdot [\Lf,\Lf]=0$ and 
$\Lf\cdot \Lf\subseteq Z(\Lf)$. By Proposition $\ref{4.2}$ we know that the result is
true for $c=3$. Let us consider $\Ln:=F_{3,c+1}$ and $\Lq:=\Ln/\Ln^{c+1}$. We have $\Lq\cong F_{3,c}$. By induction
hypothesis we have

\begin{align*}
\Lq\cdot \Lq & \subseteq \Lq^c=\Ln^c/\Ln^{c+1},\\
\Lq\cdot [\Lq,\Lq] & \subseteq \Lq^{c+1}.
\end{align*} 

This implies that

\begin{align*}
\Ln\cdot \Ln & \subseteq \Ln^c,\\
\Ln\cdot [\Ln,\Ln] & \subseteq \Ln^{c+1}.
\end{align*} 

Using \eqref{com5} we obtain

\begin{align*}
\Ln\cdot \Ln^3 & \subseteq [[\Ln\cdot \Ln,\Ln],\Ln]+[\Ln,[\Ln\cdot \Ln,\Ln]]+[\Ln,[\Ln,\Ln\cdot \Ln]]\\
               & \subseteq [[\Ln\cdot \Ln, \Ln],\Ln]\\
               & \subseteq [[\Ln^c, \Ln],\Ln] \\
               & \subseteq \Ln^{c+2}=0. 
\end{align*} 

Now \eqref{com4}, \eqref{com6} and $c\ge 3$ imply that

\begin{align*}
\Ln\cdot [\Ln,\Ln] & =[\Ln,\Ln]\cdot \Ln \\
               & \subseteq \Ln\cdot(\Ln\cdot \Ln)-\Ln\cdot (\Ln\cdot \Ln)\\
               & \subseteq \Ln\cdot(\Ln\cdot \Ln) \\
               & \subseteq \Ln\cdot \Ln^c \\
               & \subseteq \Ln\cdot \Ln^3=0. 
\end{align*} 

For all $i,j$ in $\{1,2,3\}$ we define
\[
u_{ij}:=x_i\cdot x_j \in \Ln\cdot \Ln \subseteq \Ln^c\subseteq [\Ln,\Ln].
\]
By $\Ln\cdot [\Ln,\Ln]=0$, \eqref{com4} and \eqref{com6} we obtain the equations
\begin{align*}
u_{ij} & = u_{ji},\\
[u_{ij},x_k]+[x_j,u_{ik}] & = x_i\cdot [x_j,x_k] = 0.
\end{align*}

By Corollary $\ref{4.6}$ we obtain that $u_{ij}\in \Ln^{c+1}$ for all $i,j$, so that
\begin{align*}
\Ln\cdot \Ln & \subseteq x_1\cdot \Ln+x_2\cdot \Ln +x_3\cdot \Ln+[\Ln,\Ln]\cdot \Ln \\
             & \subseteq \Ln^{c+1}.
\end{align*}

Hence the induction step is finished. \qed 
\vspace*{0.5cm}

Of course one would like to generalize Theorem $\ref{4.3}$ to all free-nilpotent Lie algebras $F_{g,c}$ 
with $g\ge 3$. The induction step works as above, but we do not know the base case for given $g\ge 4$. 

\begin{cor}
Fix an integer $g\ge 3$. Suppose that all CPA-structures on $F_{g,3}$ are central.
Then all CPA-structures $F_{g,c}$ are central for all $c\ge 3$.
\end{cor}

For $g=2$ we have established the base case, but an argument is missing to obtain a corresponding result 
to Corollary $\ref{4.5}$. However, from direct computations it is reasonable to expect that all solutions of the 
equations are central, and hence that all CPA-structures on $F_{2,c}$ are central for $c\ge 3$. 
Thus we want to pose the following conjecture.

\begin{con}\label{4.9}
All CPA-structures on $F_{g,c}$ with $c\ge 3$ and $g\ge 2$ are central.
\end{con}

For $g=2$, we have confirmed the conjecture for all $3\le c\le 10$ by direct computation.
The dimensions of  $F_{2,c}$ for $1\le c\le 10$ are $2,3,5,8,14,23,41,71,127,226$.

\section{Property F}

Motivated by CPA-structures on $F_{2,c}$ we want to find out, which $2$-generated nilpotent
Lie algebras have the property that all solutions in the commutator are central, 
for the system of equations of Example $\ref{4.7}$. 

\begin{defi}
Let $\Lg$ be a nilpotent Lie algebra with $\dim \Lg/[\Lg,\Lg]=2$. We say that $\Lg$ has
{\em property $F$}, if for every generating pair $(x,y)$ all solutions $u,v,w\in [\Lg,\Lg]$
of the two linear equations
\begin{align*}
[x,u]+[y,v] & = 0\\
[x,v]+[y,w] & = 0
\end{align*}
are contained in the center $Z(\Lg)$.
\end{defi}

We may assume that  $\Lg$ is a stem Lie algebra, because if $Z(\Lg)$ is not contained in
$[\Lg,\Lg]$, then $\Lg$ is abelian. This gives $\dim \Lg=\dim \Lg/[\Lg,\Lg]=2$, which is not
interesting. Furthermore, all $2$-generated nilpotent Lie algebras $\Lg$ with $Z(\Lg)=[\Lg,\Lg]$ 
have property $F$. \\
In general, it is not clear whether it is enough to check this property for just one specific
generating pair $(x,y)$. However, it is enough for the free-nilpotent Lie algebra $F_{2,c}$, because
in this case the automorphism group acts transitively on the pairs of generators. \\[0.2cm]
The first observation is that Lie algebras having property $F$ cannot have a small center. 

\begin{defi}
Let $\Lg$ be a Lie algebra. Denote by
\[
z (\Lg):=\frac{\dim Z(\Lg)}{\dim \Lg}
\]
the quotient of the dimension of the center by the dimension of the Lie algebra.
\end{defi}

\begin{lem}\label{5.3}
Let $\Lg$ be a nilpotent Lie algebra with $\dim \Lg/[\Lg,\Lg]=2$ and
\[
z(\Lg)<\frac{1}{3}.
\]
Then $\Lg$ does not have property $F$.
\end{lem}

\begin{proof}
Assume that $\Lg$ satisfies property $F$. We want to show that this implies $z(\Lg)\ge\frac{1}{3}$.
We may assume that $\Lg$ is a stem Lie algebra, because otherwise $\Lg$ is abelian, see above.
Since $[\Lg,\Lg]$ has codimension $2$ we may assume that our vectors $u,v,w$ are of the form
$$
(0,0,y_3,\ldots ,y_r,y_{r+1},\ldots ,y_n)
$$
with central components $\{y_{r+1},\ldots ,y_n\}$, where $n=\dim \Lg$.
We have $\dim Z(\Lg)=n-r$. The two equations are now equivalent to a system of linear equations in
the $3(r-2)$ variables of the vectors $u,v,w$. By Engel's Theorem we may assume that both
$\ad (x)$ and $\ad(y)$ are strictly lower-triangular matrices. The base change is compatible
with the condition $Z(\Lg)\subseteq [\Lg,\Lg]$. 
The above equations are just
\begin{align*}
\ad(x)u & =-\ad(y)v \\
\ad(x)v & =-\ad(y)w.
\end{align*}
Both equations correspond each to $n$ linear equations, but the first three component equations
are saying $0=0$, because $\ad (x)$ and $\ad(y)$ are strictly lower-triangular.
So we have at most $2(n-3)$ non-trivial equations. Property $F$ says that they have a unique solution, hence
\[
2(n-3)\ge 3(r-2).
\]
This implies that
\[
\dim Z(\Lg)=n-r\ge n-\frac{2n}{3}=\frac{\dim \Lg}{3}.
\]
\end{proof}

\begin{cor}\label{5.4}
A filiform nilpotent Lie algebra of dimension $n\ge 4$ does not have property $F$.
\end{cor}

\begin{proof}
A filiform nilpotent Lie algebra $\Lg$ of dimesnion $n$ satisfies $\dim \Lg/[\Lg,\Lg]=2$ 
and is a stem Lie algebra. However, we have $z(\Lg)=\frac{1}{n}<\frac{1}{3}$ for $n\ge 4$.
Hence by Lemma $\ref{5.3}$, $\Lg$ does not have property $F$.
\end{proof}

\begin{rem}
For the free-nilpotent Lie algebras with $2$ generators and nilpotency class $n$ 
let us write $f(n)=z(F_{2,n})$. Then we have
\[
f(n)=\frac{I_n}{\sum_{m=1}^n I_m},
\]
where 
\[
I_m=\frac{1}{m}\sum_{d\mid m}\mu(d)2^{m/d}.
\]
The number $I_m$ also counts the monic irreducible polynomials of degree $m$ in $\F_2$.
The first values of the sequence $(f(n))$ are given by
\[
\frac{1}{2},\frac{1}{3},\frac{3}{8},\frac{6}{14},\frac{9}{23},\frac{18}{41},\frac{30}{71},
\frac{56}{127},\frac{99}{226},\cdots
\]
We have $f(n)>\frac{1}{3}$ for all $n\ge 3$, and $I_m\sim \frac{2^m}{m}$ for $m\to \infty$. 
So we obtain 
\[
\lim_{n\to \infty} f(n)=\lim_{n\to \infty}\frac{\frac{2^n}{n}}{\sum_{m=1}^n\frac{2^m}{m}}=\frac{1}{2}.
\]
\end{rem}

Using Lemma $\ref{5.3}$ we can easily show the following result for low dimensions.

\begin{prop}
Let $\Lg$ be a complex nilpotent stem Lie algebra with $\dim \Lg/[\Lg,\Lg]=2$ and dimension $n\le 7$.
Then $\Lg$ has property $F$ if and only if $\Lg$ is isomorphic to $F_{2,2}$ or $F_{2,3}$.
\end{prop}

\begin{proof}
In dimension $3$ there is only one such Lie algebra, namely $\Lg=F_{2,2}$. It has property $F$,
since the center coincides with the commutator subalgebra. In dimension $4$ we only have the filiform nilpotent
Lie algebra $\Lg=\Ln_4$, which does not have property $F$, see Lemma $\ref{5.4}$.
In dimension $5$, the only stem nilpotent Lie algebra with $\dim \Lg/[\Lg,\Lg]=2$ and $\dim Z(\Lg)\ge 2$ is
$F_{2,3}$. It is easy to verify that $F_{2,3}$ has property $F$, since it is enough to check it just for
one specific pair of generators. \\
In dimension $6$, the classification shows that the only stem nilpotent Lie algebra satisfying
$\dim \Lg/[\Lg,\Lg]=2$ and $\dim Z(\Lg)\ge 2$ is given by $\Lg_{6,14}$, in Magnin's notation \cite{MAG}.
The Lie brackets are given by
\[
[x_1,x_i]=x_{i+1}, 2\le i\le 4,[x_2,x_3]=x_6.
\]
Here $x_1,x_2$ generate $\Lg$. A direct computation shows that is does not have property $F$. \\
In dimension $7$, Magnin's list \cite{MAG} shows that there is no nilpotent stem Lie algebra 
with $\dim \Lg/[\Lg,\Lg]=2$ and $\dim Z(\Lg)\ge 3$.
\end{proof}

In general, it seems likely that the following holds true:

\begin{prob*}
A nilpotent Lie algebra $\Lg$ with $\dim \Lg/[\Lg,\Lg]=2$ has property $F$ if and only
if it is isomorphic to some $F_{2,c}$.
\end{prob*}

\section*{Acknowledgements}
Dietrich Burde is supported by the Austrian Science Foun\-da\-tion FWF, grant P28079 
and grant I3248. Wolfgang A. Moens acknowledges support by the Austrian Science Foun\-da\-tion FWF, 
grant P28079 and P30842. Karel Dekimpe is supported by  long term structural funding -- Methusalem 
grant of the Flemish Government.


\begin{thebibliography}{99}


\bibitem{BU5} D. Burde: {\it Affine structures on nilmanifolds}. International Journal of
Mathematics \textbf{7} (1996), no. 5, 599--616.

\bibitem{BU19} D. Burde, K. Dekimpe, S. Deschamps: {\it The Auslander conjecture for NIL-affine
crystallographic groups}. Mathematische Annalen \textbf{332} (2005), no. 1, 161--176.

\bibitem{BU24} D. Burde: {\it Left-symmetric algebras, or pre-Lie algebras in geometry
and physics}. Central European Journal of Mathematics \textbf{4} (2006), no. 3, 323--357.

\bibitem{BU33} D. Burde, K. Dekimpe and S. Deschamps: {\it Affine actions on nilpotent
Lie groups}. Forum Math.\ \textbf{21} (2009), no. 5, 921--934.

\bibitem{BU34} D. Burde, K. Dekimpe and S. Deschamps: {\it LR-algebras}.
Contemporary Mathematics \textbf{491} (2009), 125--140.

\bibitem{BU38} D. Burde, K. Dekimpe, K. Vercammen: {\it Complete LR-structures on solvable
Lie algebras}. Journal of Group Theory \textbf{13} (2010), no. $5$, 703--719.

\bibitem{BU41} D. Burde, K. Dekimpe and K. Vercammen: {\it Affine actions on Lie groups
and post-Lie algebra structures}.
Linear Algebra and its Applications \textbf{437} (2012), no. 5, 1250--1263.

\bibitem{BU44} D. Burde, K. Dekimpe: {\it Post-Lie algebra structures and generalized
derivations of semisimple Lie algebras}.
Moscow Mathematical Journal, Vol. \textbf{13} (2013), Issue 1, 1--18.

\bibitem{BU51} D. Burde, K. Dekimpe: {\it Post-Lie algebra structures on pairs of Lie algebras}.
Journal of Algebra \textbf{464}(2016), 226--245.

\bibitem{BU52} D. Burde, W. A. Moens: {\it Commutative post-Lie algebra structures on Lie algebras}.
Jorunal of Algebra \textbf{467} (2016), 183--201.

\bibitem{HEL} J. Helmstetter: {\it Radical d'une alg\`ebre sym\'etrique
a gauche}. Ann.\ Inst.\ Fourier \textbf{29} (1979), 17--35.

\bibitem{ELM} K. Ebrahimi-Fard, A. Lundervold, I. Mencattini, H. Z. Munthe-Kaas:  
{Post-Lie Algebras and Isospectral Flows}.
SIGMA Symmetry Integrability Geom.\ Methods Appl.\ \textbf{11} (2015), Paper 093, 16 pp.

\bibitem{KIM} H. Kim: {\it Complete left-invariant affine structures on nilpotent Lie groups}.
J. Differential Geom.  \textbf{24} (1986), no. 3, 373--394.


\bibitem{LOD} J.-L. Loday: {\it Generalized bialgebras and triples of operads}.
Astérisque  No. \textbf{320} (2008), 116 pp. 

\bibitem{MAG} L. Magnin: {\it Determination of 7-dimensional indecomposable nilpotent complex Lie
algebras by adjoining a derivation to 6-dimensional Lie algebras}.
Algebras and Representation Theory \textbf{13}, no. 6 (2010), 723--753.

\bibitem{RES}  V.  Remeslennikov, R. St\"ohr: {\it The equation $[x,u]+[y,v]=0$ in free Lie algebras}.
Internat.\ J.\ Algebra Comput.\ \textbf{17} (2007), no. 5-6, 1165--1187. 

\bibitem{SEG} D. Segal: {\it The structure of complete left-symmetric algebras}.
Math.\ Ann.\ \textbf{293} (1992), 569--578.

\bibitem{VAL} B. Vallette: {\it Homology of generalized partition posets}.
J.\ Pure and Applied Algebra \textbf{208} (2007), no. 2, 699--725. 

\end{thebibliography}
\end{document}